\documentclass{amsart}

\usepackage{amsmath,amssymb,enumerate}
\usepackage{epsfig}

\newtheorem{theorem}{Theorem}

\newtheorem{conjecture}{Conjecture}
\newtheorem{example}{Example}
\newtheorem{proposition}{Proposition}
\newtheorem{lemma}[proposition]{Lemma}
\newtheorem{corollary}[proposition]{Corollary}
\newtheorem{fact}[proposition]{Fact}

\newtheorem{definition}[proposition]{Definition}
\newtheorem{remark}[proposition]{Remark}

\newcommand{\occult}[1]{}

\newcommand{\new}[1]{{\bf #1}}

\newcommand\eps{\epsilon}

\newcommand\lip{{\operatorname{lip}}}
\newcommand\Lip{{\operatorname{Lip}}}
\newcommand\NN{{\mathbb N}}
\newcommand\RR{{\mathbb R}}
\renewcommand\top{{\operatorname{top}}}
\newcommand\TT{{\mathbb T}}
\newcommand\vol{{\operatorname{vol}}}
\newcommand\ZZ{{\mathbb Z}}

\begin{document}

\title{Dimensional Entropies and\\
Semi-Uniform Hyperbolicity}
\date{Talk at ICMP - Rio de Janeiro, Brazil, August 2006}

\begin{abstract}
We describe \emph{dimensional entropies} introduced in \cite{BuzziSMF},
list some of their properties, giving some proofs.
These entropies allowed the definition in \cite{BuzziEE,BuzziSaintFlour} 
of \emph{entropy-expanding maps}. We introduce a new notion of
\emph{entropy-hyperbolicity} for diffeomorphisms. We indicate some
simple sufficient conditions (some of them new) for these properties.
We conclude by some work in progress and more questions.
%
\end{abstract}

\author{J\'er\^ome Buzzi}

\maketitle


\section{Introduction}

We are interested in using robust entropy conditions to study chaotic dynamical
systems. These entropy conditions imply some "semi-uniform" hyperbolicity. This is
a type of hyperbolicity which is definitely weaker than classical uniform hyperbolicity
but which is stronger than Pesin hyperbolicity, that is, non vanishing of the
Lyapunov exponents of some relevant measure. This type of conditions allows the
generalization of some properties of interval maps 
and surface diffeomorphisms to arbitrary dimensions.

In this paper, we first explain what is known in low dimension just assuming the
non-vanishing of the topological entropy $h_\top(f)$. Then we give a detailed description of the
\emph{dimensional entropies}. These are $d+1$ numbers, if $d$ is the dimension
of the manifold,
 $$
  0 = h^0_\top(f)\leq h^1_\top(f)\leq\dots\leq h^d_\top(f)=h_\top(f).
 $$ 
$h^k_\top(f)$ "counts" the number of orbits starting from an arbitrary compact and smooth
$k$-dimensional submanifold.
We both recall known properties and establish new ones. We then recall the definition of \emph{entropy-expanding maps} which
generalize the complexity of interval dynamics with non-zero topological entropy. We also introduce
a similar notion for diffeomorphisms:

\begin{definition}
A diffeomorphism of a $d$-dimensional manifold is \new{entropy-hyperbolic} if
there are integers $d_u,d_s$ such that:
 \begin{itemize}
  \item $h^{d_u}_\top(f)=h_\top(f)$ and this fails for every dimension $k<d_u$;
  \item $h^{d_s}_\top(f^{-1})=h_\top(f)$ and this fails for every dimension $k<d_s$;
  \item $d_u+d_s=d$.
 \end{itemize}
\end{definition}

We give simple sufficient conditions for entropy-expansion and entropy-hyperbolicity.
Finally we announce some work in progress and state a number of questions.

\medbreak

We now recall some classical notions which may be found in \cite{KH}. 

A basic measure of orbit complexity of a map $f:M\to M$ is the \emph{entropy}. 
The \emph{topological entropy} $h_\top(f)$  "counts" all the orbits
and the \emph{measure-theoretic entropy} (also known as Kolmogorov-Sinai entropy
or ergodic entropy) $h(f,\mu)$ "counts" the orbits "relevant" to some given invariant
probability measure $\mu$. They are related by the following rather general variational 
principle. If, e.g., $f$ is continuous and $M$ is compact, then
 $$
    h_\top(f) = \sup_\mu h(f,\mu)
 $$
where $\mu$ ranges over all invariant probability measures. One can also restrict
$\mu$ to \emph{ergodic} invariant probability measures. 

This brings to the fore measures which realize the above supremum, when they
exist, and more generally measures which have entropy close to this supremum.

As $\mu\mapsto h(f,\mu)$
is affine, $\mu$ has maximum entropy if and only if almost every ergodic component
of it has maximum entropy. Hence, with respect to entropy, it is enough to study 
ergodic measures.

\begin{definition}
A {\bf maximum measure} is an ergodic and invariant probability measure $\mu$ such
that $h(f,\mu)=\sup_\nu h(f,\nu)$.

A {\bf large entropy measure} is  an ergodic and invariant probability measure $\mu$ such
that $h(f,\mu)$ is close to $\sup_\nu h(f,\nu)$.
\end{definition}

The \emph{Lyapunov exponents} for some ergodic and invariant probability measure $\mu$
are the possible values $\mu$-a.e. of the limit $\lim_{n\to\infty}\frac1n\log\|T_xf^n.v\|$
where $\|\cdot\|$ is some Riemannian structure and $T_xf$ is the differential of
$f$ and $v$ ranges over the non-zero vectors of the tangent space $T_xM$.

A basic result connecting entropy and hyperbolicity is
the following theorem (proved by Margulis for volume preserving flows):

\begin{theorem}[Ruelle's inequality]
Let $f:M\to M$ be a $C^1$ map on a compact manifold. 
Let $\mu$ be an $f$-invariant ergodic probability measure.
Let $\lambda_1(\mu)\geq\dots$ be its Lyapunov exponents 
repeated according to multiplicity. Then,
 $$
   h(f,\mu) \leq \sum_{i=1}^d \lambda_i(\mu)^+
 $$
\end{theorem}
In good cases (with enough hyperbolicity), the entropy is also reflected in
the existence of many periodic orbits:

\begin{definition}
The periodic points of some map $f:M\to M$ satisfy a {\bf multiplicative lower bound}, 
if, for some integer $p\geq1$:
 $$
   \liminf_{n\to\infty,p|n} e^{-n h_\top(f)} \#\{x\in[0,1]: f^nx=x\} > 0.
 $$
\end{definition}

Recall that many diffeomorphisms have infinitely many more periodic orbits (see
\cite{Kaloshin,KaloshinHunt}).

The following type of isomorphism will be relevant to describe all
"large entropy measures".

\begin{definition}
For a given measurable dynamical system $f:M\to M$, a subset $S\subset M$ is
\new{entropy-negligible} if there exists $h<\sup_\mu h(f,\mu)$ such that for
all ergodic and invariant probability measures $\mu$ with $h(f,\mu)>h$,
$\mu(S)=0$.

An {\bf entropy-conjugacy} between two measurable dynamical systems $f:M\to M$
and $g:N\to N$ is a bi-measurable invertible mapping $\psi:M\setminus M_0\to N\setminus
N_0$ such that: $\psi$ is a conjugacy (i.e., $g\circ\psi=\psi\circ f$) and 
$M_0$ and $N_0$ are entropy-negligible.
\end{definition}

\section{Low Dimension}

Low dimension dynamical systems here means interval maps
and surface diffeomorphisms - those systems for which non-zero
entropy is enough to ensure hyperbolicity of the large entropy
measures.

\subsection{Interval Maps}

Indeed, an immediate consequence of Ruelle's inequality on the interval 
is that a lower bound on the measure-theoretic entropy gives a lower bound on the (unique)
Lyapunov exponent. Thus, invariant measures with nonzero topological
entropy are hyperbolic in the sense of Pesin. One can obtain much more from the
topological entropy:

\begin{theorem}
Let $f:[0,1]\to[0,1]$ be $C^\infty$. If $h_\top(f)>0$
then $f$ has finitely many maximum measures. Also the periodic points satisfy a multiplicative lower
bound.
\end{theorem}

This was first proved by F. Hofbauer \cite{Hofbauer1,Hofbauer2} for
piecewise monotone maps (admitting finitely many points $a_0=0<a_1<
\dots<a_N$ such that $f|]a_i,a_{i+1}[$ is continuous and monotone).
It was then extended to arbitrary $C^\infty$ maps in \cite{BuzziSIM}.
In both settings, one builds an entropy-conjugacy to a combinatorial 
model called a Markov shift (which is a subshift of finite type over
an infinite alphabet). One can then apply some results of D. Vere-Jones
\cite{VJ} and B. Gurevi\v{c} \cite{Gurevic1}.

\medbreak

We can even classify these dynamics. Recall that the \emph{natural extension}
of $f:M\to M$ is $\bar f:\bar M\to\bar M$ defined as $\bar M:=\{(x_n)_{n\in\ZZ}\in
M^\ZZ:\forall n\in\ZZ\; x_{n+1}=f(x_n)\}$ and $\bar f((x_n)_{n\in\ZZ})=(f(x_n))_{n\in\ZZ}$.
Recall that $\bar\pi:(x_n)_{n\in\ZZ}\mapsto x_0$ induces a homeomorphism 
between the spaces of invariant probability measures which respects entropy and ergodicity.

\begin{theorem}
The natural extensions of $C^\infty$ interval maps with
non-zero topological entropy are classified up
to entropy-conjugacy by their topological entropy and finitely many
integers (which are "periods" of the maximum measures).
\end{theorem}

The classification is deduced from the proof of the previous theorem
by using a classification result \cite{BBG} for the invertible Markov shifts
involved.
 
\medbreak

The $C^\infty$ is necessary: for each finite $r$,
there are $C^r$ interval maps with non-zero topological entropy having
infinitely many maximum measures and others with none.

\begin{remark}
These examples show  in particular that Pesin hyperbolicity of maximum measures or even of
large entropy measures (which are both consequences of Ruelle's
inequality here) are not enough to ensure the finite number of maximum
measures.
\end{remark}

\subsection{Surface Transformations}

As observed by Katok \cite{KatokPeriodic}, Ruelle's inequality
applied to a surface diffeomorphism and its inverse (which has
opposite Lyapunov exponents) shows that a lower-bound on measure-theoretic
entropy bounds away from zero the Lyapunov exponents of the measure.
Thus, for surface diffeomorphisms also, nonzero entropy implies
Pesin hyperbolicity.

It is believed that surface diffeomorphisms should behave as
interval maps, leading to the following folklore conjecture:

\begin{conjecture}\label{conj:surface}
Let $f:M\to M$ be a $C^\infty$ surface diffeomorphism. If
$h_\top(f)>0$ then $f$ has finitely many maximum measures.
\end{conjecture}

I would think that, again like for interval maps, finite
smoothness is not enough for the above result. However
counter-examples to this (or to existence) are known only
in dimension $\geq4$ \cite{Misiurewicz}.

The best result for surface diffeomorphisms at this point is
the following "approximation in entropy" \cite{KatokPeriodic}:

\begin{theorem}[A. Katok]\label{thm:Katok}
Let $f:M\to M$ be a $C^{1+\eps}$ surface diffeomorphism. For
any $\eps>0$, there exists a horseshoe\footnote{A horseshoe is an invariant
compact subset on which some iterate of $f$ is conjugate with a full shift on
finitely many symbols.} $\Lambda\subset M$ such
that $h_\top(f|\Lambda)>h_\top(f)-\eps$. In particular, the
periodic points of $f$ satisfy a logarithmic lower bound:
 $$
   \limsup_{n\to\infty} \frac1n\log\#\{x\in M:f^n(x)=x\} \geq h_\top(f).
 $$
\end{theorem}

Katok in fact proved a more general fact, valid for any $C^{1+\eps}$-diffeomorphism
of a compact manifold of any dimension. Namely, if $\mu$ is an ergodic invariant probability
measure without zero Lyapunov exponent :
 $$
  \limsup_{n\to\infty} \frac1n\log\#\{x\in M:f^n(x)=x\} \geq h(f,\mu).
 $$
On surfaces, Ruelle's inequality and the variational principle imply the theorem as explained
above.

\medbreak

I have proved the conjecture for a model class, which replaces
distortion with (simple) singularities \cite{BuzziPWAH}:

\begin{theorem}
Let $f:M\to M$ be a piecewise affine homeomorphism. If $h_\top(f)>0$
then $f$ has finitely many maximum measures.
\end{theorem}

\section{Dimensional Entropies}

We are going to define the dimensional entropies for a smooth self-map or 
diffeomorphism $f:M\to M$ of a $d$-dimensional compact manifold.
We will then investigate these quantities by considering
other growth rates obtained from the volume and size of the derivatives.
Finally we shall establish the topological variational principle stated in
the introduction by a variant of Pliss Lemma.

\subsection{Singular disks}

The basic object is:

\begin{definition}
A \new{(singular) $k$-disk} is a map $\phi:Q^k\to M$ with $Q^k:=[-1,1]^k$. 
It is $C^r$ if it can be extended to a $C^r$ map on a neighborhood of $Q^k$.
\end{definition}

We need to define the $C^r$ size $\|\phi\|_{C^r}$ of a singular disk $\phi$
for $1\leq r\leq\infty$ as well as the corresponding topologies on the space of such disks.
This involves some technicalities as vectors in different tangent spaces are
not comparable {\it a priori}. We refer to Appendix \ref{appendix:Cr-size}
for the precise definitions, which are rather obvious for finite $r$. For $r=\infty$,
we need an approximation property (which fails for some otherwise very reasonable definitions
of $C^r$ size), Fact \ref{fact:Cinfty-approx}, which is used to prove  Lemma \ref{lem:h-Cinfty} below. 

From now on, we fix some $C^r$
size arbitrarily on the manifold $M$. We will later check that the entropies we are
interested in are in fact independent of this choice.

\medbreak

\noindent{\sl Notations.}
It will be convenient to sometimes write $\phi$ instead of $\phi(Q^k)$, e.g., $h_\top(f,\phi)$
instead of $h_\top(f,\phi(Q^k))$.

\subsection{Entropy of collections of subsets}

Given a collection $\mathcal D$ of subsets of $M$, we associate the following entropies.
Recall that the $(\eps,n)$-covering number of some subset $S\subset M$ is:
 $$
   r_f(\eps,n,S):=\min\{\#C:\bigcup_{x\in S} B_f(\eps,n,x) \supset S\}
 $$
where $B_f(\eps,n,x):=\{y\in M:\forall 0\leq k<n\; d(f^ky,f^kx)<\eps \}$ is the $(\eps,n)$-dynamic
ball. The classical Bowen-Dinaburg formula for the topological entropy of $S\subset M$ is
$h_\top(f,S)=\lim_{\eps\to0}\limsup_{n\to\infty}\frac1n\log r_f(\eps,n,S)$ and $h_\top(f)=
h_\top(f,M)$.

\begin{definition}
The \new{topological entropy} of $\mathcal D$ is:
 $$
    h_\top(f,\mathcal D) := \sup_{D\in\mathcal D} h_\top(f,D)
      = \sup_{D\in\mathcal D} \lim_{\eps\to0} \limsup_{n\to\infty} \frac1n\log r_f(\eps,n,D)
 $$
The \new{uniform topological entropy} of $\mathcal D$ is:
 $$
    H_\top(f,\mathcal D) := 
     \lim_{\eps\to0} \limsup_{n\to\infty} \frac1n\log \sup_{D\in\mathcal D} r_f(\eps,n,D).
 $$
\end{definition}

Clearly $h_\top(f,\mathcal D)\leq H_\top(f,\mathcal D)$. The inequality can be strict as 
shown in the following examples (the first one involving non-compactness, the second one
involving non-smoothness).

\begin{example}
Let $T:\TT^2\to\TT^2$ be a linear endomorphism with two eigenvalues $\Lambda_1,\Lambda_2$ with $1<|\Lambda_1|
<|\Lambda_2|$. Let $\mathcal L$
be the set of finite line segments. We have
 $$
   0 < h_\top(T,\mathcal L)=\log|\Lambda_2| < H_\top(T,\mathcal L) = \log|\Lambda_1| +\log|\Lambda_2|.
 $$ 
\end{example}

\begin{example}
There exist a $C^\infty$ self-map $F$ of $[0,1]^2$ and a collection $\mathcal C$ of $C^r$ curves with 
bounded $C^r$ norm such that $0 < h_\top(F,\mathcal C) < H_\top(F,\mathcal C)$. This can be deduced from the 
example with $h^1_\top(f\times g)>\max(h_\top(f),h_\top(g))$ in \cite{BuzziSMF} by considering curves 
with finitely many bumps converging $C^{r-1}$ to the example curve there, which has infinitely many bumps. 
\end{example}

\subsection{Definitions of the dimensional entropies}

We can now properly define the dimensional entropies. Recall that we have endowed $M$ with
a $C^r$ size.

\begin{definition}
For each $1\leq r\leq\infty$, the \new{standard family} of $C^r$ singular $k$-disks is the collection
of all $C^r$ singular $k$-disks. For finite $r$, the \new{standard uniform family} of $C^r$ singular 
$k$-disks is the collection of all $C^r$ singular $k$-disks with $C^r$ size bounded by $1$.
\end{definition}

\begin{definition}
The \new{$C^r$, $k$-dimensional entropy} of a self-map $f$ of a compact manifold is:
 $$
     h^{k,C^r}_\top(f) := h_\top(f,\mathcal D^k_r)
 $$
where $\mathcal D^k_r$ is a standard family of $C^r$ $k$-disks of $M$. We write
$h^k_\top(f)$ for $h^{k,C^\infty}(f)_\top$ and call it the $k$-dimensional entropy. 

The \new{$C^r$, $k$-dimensional uniform entropy} $H^{k,C^r}_\top(f)$ is obtained 
by replacing $h_\top(f,\mathcal D^k_r)$ with $H_\top(f,D^k_r)$ in the 
above definition where $D^k_r$ is the standard uniform family. We write $H^k_\top(f)$ for $H^{k,C^\infty}_\top(f)$ and call it the $k$-dimensional 
uniform entropy.
\end{definition}

Observe that $h^{k,C^r}_\top(f)$ and $H^{k,C^r}_\top(f)$ are non-decreasing functions of $k$ and non-increasing
functions of $r$. Indeed, (1) $\mathcal D^k_r\supset\mathcal D^k_s$ and $D^k_r\supset D^k_s$ if $r\leq s$;
(2) for any $0\leq\ell\leq
k\leq d$, restricting a $k$-disk to $[0,1]^\ell\times\{0\}^{k-\ell}$ does not increase its
$C^r$ size. Observe also that $h^{0,C^r}_\top(f)=0$ and $h^d_\top(f)=h_\top(f)$.

\begin{lemma}\label{lem:h-Cinfty}
Let $f:M\to M$ be a $C^\infty$ self-map of a compact manifold. We have:
 $$
   H^k_\top(f) = \lim_{r\to\infty} H^{k,C^r}_\top(f)
 $$
and the limit is non-increasing.
\end{lemma}

We shall see later in Proposition \ref{prop:h-and-H} that the same holds for $h^k_\top(f)$.

\begin{proof}
We use one of the (simpler) ideas of Yomdin's theory. For each $n\geq1$,
we divide $Q^k$ into small cubes with diameter at most $(\eps/4)^{1/r}\|\phi\|_{C^r}^{-1/r}\Lip(f)^{-n/r}$. 
We need $(\eps/4)^{-k/r}\sqrt{k}^k \|\phi\|_{C^r}^{k/r}\Lip(f)^{\frac kr n}$ such cubes. Let $q$ be
one of them. By Fact \ref{fact:Cinfty-approx}, there exists a $C^\infty$ $k$-disk $\phi_q$ such that
$\|\phi_q\|_{C^\infty}\leq2\|\phi\|_{C^r}$ and
 $$
   \forall t\in q\quad d(\phi_q(t),\phi(t))\leq \|\phi\|_{C^r} \|t-t_q\|^r
     \leq  \|\phi\|_{C^r} \times \frac\eps2 \|\phi\|_{C^r}^{-1}\Lip(f)^{-n}
     \leq \frac\eps2 \Lip(f)^{-n}
 $$
It follows that $r_f(\eps,n,\phi\cap q)\leq r_f(\eps/2,n,\phi_q)$. Thus,
 $$
    r_f(\eps,n,\mathcal D^k_r) \leq \sqrt{k}^k(\eps/4)^{-k/r} \|\phi\|_{C^r}^{k/r}\Lip(f)^{\frac kr n} r_f(\eps/2,n,\mathcal D^k_\infty)
 $$
Hence, writing $\lip(f):=\max(\log\Lip(f),0)$,
 $$
    H^{k,C^r}_\top(f) \leq \frac kr \lip(f) + H^k_\top(f).
 $$
The inequality $H^{k,C^r}_\top(f) \geq H^k_\top(f)$ is obvious, concluding the proof.
\end{proof}

\begin{lemma}
The  numbers $H^{k,C^r}_\top(f)$ do not depend on the underlying choice of a $C^r$ size.
\end{lemma}

\begin{proof}
Using Lemma \ref{lem:h-Cinfty}, it is enough to treat the case with finite smoothness.
Let $\mathcal D_1,\mathcal D_2$ be two standard families of $k$-disks, defined by two $C^r$ sizes
$\|\cdot\|_{C^r}^1,\|\cdot\|_{C^r}^2$. By Fact \ref{fact:equiv-Cr}, there exists $C<\infty$ such that
$\|\cdot\|_{C^r}^1\leq C\|\cdot\|_{C^r}^2$. Hence setting $K:=([C]+1)^k$, for any $k$-disk 
$\phi_1\in\mathcal D_1$ can be \emph{linearly
subdivided}\footnote{That is, each $\phi_2^j=\phi_1\circ L_j$ with $L_j:Q^k\to Q^k$ linear and
$\bigcup_{j=1}^K L_j(Q^k)=Q^k$.} into $K$ $k$-disks $\phi_2^1,\dots,\phi_2^K\in\mathcal D_2$.
Thus
 $$\forall n\geq0\quad r_f(\eps,n,\phi_1)\leq K\max_j r_f(\eps,n,\phi_2^j).$$
It follows immediately that $H(f,\mathcal D_1) 
\leq H(f,\mathcal D_2)$. The claimed equality follow in turn by symmetry.
\end{proof}

\section{Other growth rates of submanifolds}

\subsection{Volume growth}

Entropy is a growth rate under iteration. Equipping $M$ with a Riemannian structure 
allows the definition of volume growth of submanifolds.

\begin{definition}
Let $\phi:Q^k\to M$ be a singular $k$-disk. Its \new{(upper) growth rate} is:
 $$
   \gamma(f,\phi):=\limsup_{n\to\infty} \frac1n\log \vol(f^n\circ\phi) \text{ with }
    \vol(\psi):=\int_{Q^k} \| \Lambda^k\psi(x)\| \, dx
 $$
where $\|\Lambda^k\psi\|$ is the Jacobian of $\psi:Q^k\to M$ wrt the obvious Riemannian structures.
The volume growth exponent of $f$ in dimension $k$ is:
 $$
    \gamma^k(f):= \sup_{\phi\in\mathcal D^k_r} \gamma(f,\phi),
 $$
$\gamma(f):=\max_{0\leq k\leq d}\gamma^k(f)$ is simply called the \new{volume growth} of $f$.
\end{definition}

Observe that the value of the growth rates defined above are independent of the choice of the Riemannian structure, by compactness of the manifold.

The volume growth dominates the entropy quite generally:

\begin{theorem}[Newhouse \cite{NewhouseVol}]
Let $f:M\to M$ be a $C^{1+\alpha}$, $\alpha>0$, smooth self-map of a compact
manifold. Then:
 $$
    h_\top(f) \leq \gamma(f).
 $$
\end{theorem}

\begin{remark}
More precisely, his proof gave
 $$
   h_\top(f) \leq \gamma^{d^{cu}}(f)
 $$
where $d^{cu}$ is such that the variational principle $h_\top(f)=\sup_\mu h(f,\mu)$ still holds when
$\mu$ is restricted to measures with exactly $d^{cu}$ nonpositive Lyapunov exponents.

For $C^{1+1}$-diffeomorphisms, deeper ergodic techniques due to Ledrappier and Young are available 
and Cogswell \cite{Cogswell} has shown that, for any ergodic invariant probability measure $\mu$, there 
exists a disk $\Delta$ such that $h_\top(f,\mu)\leq h_\top(f,\Delta)\leq \gamma(f,\Delta)$.
More precisely, the dimension of this disk is the number of positive Lyapunov exponents. For $C^\infty$
diffeomorphisms (more generally if there is a maximum measure) there exists a disk $\Delta_{\max}$ such that $$h_\top(f)=h_\top(f,\Delta_{\max})=\gamma(f,\Delta_{\max}).$$

I do not know if Newhouse's inequality fails for $C^1$ maps.
\end{remark}

The proof of Newhouse inequality involves \emph{ergodic theory} and especially \emph{Pesin theory}. 
Indeed, this type of inequality does not hold uniformly:

\begin{example}
There exist a $C^\infty$ self-map $F$ of a surface and a 
$C^\infty$ curve $\phi$ such that, for some sequence $n_i\to\infty$,
 $$
    \lim_{i\to\infty} \frac1{n_i}\log r_F(\eps,n_i,\phi) > 
    \lim_{i\to\infty} \frac1{n_i}\log \vol(F^{n_i}\circ\phi).
 $$
\end{example}

\begin{proof}
Let $\alpha>0$ be some small number. Let $I:=[0,1]$. Let $f:I\to I$ be a $C^\infty$
map such that: (i) $f(0)=f(1)=0$; (ii) $f(1/2)=1$; (iii) $f|[0,1/2]$ is increasing
and $f|[1/2,1]$ is decreasing; (iv) $f'|[0,1/2-\alpha]=2(1+\alpha)$ and $f'|[1/2+\alpha]=
-2(1+\alpha)$. 
As $\alpha$ is small, $1/2$ has a preimage in  $[0,1/2-\alpha]$. 
Let $x_{-n}$ be the leftmost preimage in $f^{-n}(1)$: $x_0=1$, $x_{-1}=1/2$,
and for all $n\geq2$, $x_{-n}=2^{-n}(1+\alpha)^{-n+1}$.
Let $g:I\to I$ be another $C^\infty$ map such that: (i) $g(0)=0$; (ii) $0<g'<1$;
(iii) $g(x_{-n})=x_{-n-1}$ for all $n\geq0$.

Consider the following composition of length $3n$ for some $n\geq1$:
 \begin{equation}\label{eq:composition-1}
    I \longrightarrow^{f^n} 2^n\times I \longrightarrow^{g^n} 2^n \times [0,x_{-n}]
    \longrightarrow^{f^n} 2^n\times I.
 \end{equation}
Observe that after time $2n$, the length of the curve $g^n\circ f^n$ is
$2^n\cdot x_{-n}=(1+\alpha)^{-n+1}$ whereas the number of $(\eps,n)$-separated orbits is
less than $\eps^{-1}n2^n$. After time $3n$, the curve $f^n\circ g^n\circ f^n$ has
image $I$ with multiplicity $2^n$. It is therefore easy to analyze the dynamics of 
compositions of such sequences.

We build our example by considering a skew-product for which the curve will be a fiber
over a point which will drive the application of sequences as above.

Let $h:S^1\to S^1$ be the circle map defined by $h(\theta)=4\theta\mod 2\pi$.
Let $F:S^1\times I\to S^1\times I$ be a $C^\infty$ map such that:
$F(\theta,x)=(h(\theta),f(x))$ if $\theta\in[0,\frac16]$ and
$F(\theta,y)=(h(\theta),g(x))$ if $\theta\in[\frac12,\frac23]$

Recall that the expansion in basis $4$ of $\theta\in S^1$ is the sequence
$a_1a_2\dots\in\{0,1,2,3\}^\NN$ such that $\theta=2\pi\sum_{k\geq1} a_k 4^{-k}$.
We write $\theta=\overline{0.a_1a_2a_3\dots}^4$.

Observe that, whenever $\theta$ has only $0$s and $2$s in its expansion, 
 $$h^n(\theta)\in[0,\frac16] =[0,\overline{0.02222\dots}^4]$$
whenever its $n$th digit is $0$ and 
 $$h^n(\theta)\in[\frac12,\frac23]=[\overline{0.2000\dots}^4,
\overline{0.222\dots}^4]
 $$ 
whenever its $n$th digit is $2$.
Thus we can specify the desired compositions of $f$ and $g$ just by picking $\theta\in S^1$
with the right expansion. We pick:
 $$
    \theta_1 = \overline{0.0^{n_1}2^{n_1}0^{n_1+n_2}2^{n_2}0^{n_2+n_3}2^{n_3}0^{n_3+n_4}....}^4
 $$
so that we shall have a sequence of compositions of the type (\ref{eq:composition-1}). We write $N_i:=3n_1+\dots+3n_i$. We set $n_i:=i!$ so that 
$n_{i+1}/N_i\to\infty$. Let $\phi^1:Q^1\to S^1\times I$ be defined by $\phi^1(s)=(\theta_1,
(s+1)/2)$.

The previous analysis shows that $F^{N_i}\circ\phi^1$ has image $I$ with
multiplicity $2^{n_1+\dots+n_i}=2^{\frac13N_i}$. $F^{N_i+n_{i+1}}\circ\phi^1$
has image $I$ with multiplicity $2^{\frac13N_i}\times 2^{n_{i+1}}$. $F^{N_i+n_{i+1}}\circ\phi^1$
has image $[0,x_{-n_{i+1}}]$ with multiplicity $2^{\frac23N_i}\times 2^{n_{i+1}}$. It follows
that, setting $t_i:=N_i+2n_{i+1}\approx 2n_{i+1}$,
 $$
   \log r_F(\eps,t_i,\phi^1) \approx (\frac13 N_i+n_{i+1})\log 2 
 $$
whereas
 $$
   \vol(F^{t_i}\circ\phi^1) = x_{-n_{i+1}}\times 2^{\frac13N_i+n_{i+1}} =
   (1+\alpha)^{-n_{i+1}} 2^{\frac13N_i}.
 $$
Hence,
 $$
   \frac1{t_i} \log r_F(\eps,t_i,\phi^1) \approx \frac12\log 2 \text{ whereas }
  \frac1{t_i} \log \vol(F^{t_i}\circ\phi^1)\preceq -\frac12\alpha,
 $$
as claimed.
\end{proof}

\begin{remark}
The inequality in the previous example is obtained as the
length is contracted after a large expansion. For curves, this is in
fact general and it is
easily shown that, for any $C^1$ $1$-disk $\phi$ with unit length, for any 
$0<\eps<1$:
 \begin{equation}\label{eq:counterexample2}
    \forall n\geq0\quad \eps \cdot r_f(\eps,n,\phi) \leq \max_{0\leq k< n} \vol(f^k\circ\phi) + 1.
 \end{equation}
(\ref{eq:counterexample2}) implies that, for curves,
 \begin{equation}\label{eq:htop-less-vol}
   h_\top(f,\phi)\leq\gamma(f,\phi),
  \end{equation} 
both quantities being defined by $\limsup$ (this would fail using $\liminf$).
However, one can find similarly as above, a $C^\infty$ self-map of a $3$-dimensional compact manifold 
and a $C^\infty$ smooth $2$-disk such that (\ref{eq:counterexample2}) fails though
(\ref{eq:htop-less-vol}) seems to hold.
\end{remark}

We ask the following:

\begin{question}
{\it Let $f:M\to M$ be a $C^\infty$ self-map of a compact $d$-dimensional manifold. 
Is it true that, for any singular $k$-disk $\psi$ ($0\leq k\leq d$)
 $$
   h_\top(f,\psi)\leq\max_{\phi\subset\psi} \gamma(f,\phi)\quad ?
 $$ 
(both rates being defined using $\limsup$ and $\phi$ ranging over singular $\ell$-disks, $0\leq\ell<k$, with
$\phi(Q^\ell)\subset\psi(Q^k)$)? Is it at least true that 
 $$
   h^k_\top(f) \leq \max_{0\leq \ell\leq k} \gamma^\ell(f)\quad ?
 $$
These might even hold for finite smoothness for all I know.}
\end{question}

Conversely, entropy also provides some bounds on volume growth

\begin{theorem}[Yomdin \cite{Yomdin1}]
Let $f:M\to M$ be a $C^r$, $r\geq1$, smooth self-map of a compact
manifold. Let $\alpha>0$. Then there exist $C(r,\alpha)<\infty$ and $\eps_0(r)>0$ with the following
property. Let $\phi:Q^k\to M$ be any $C^r$ singular $k$-disk with unit $C^r$ size for some
$0\leq k\leq d$. Then, for any $n\geq0$,
 $$
    \vol(f^n\circ\phi) \leq C(r,\alpha) \Lip(f)^{(\frac kr+\alpha) n} \cdot r_f(\eps_0,n,\phi).
 $$
In particular,
 $$
   \gamma^k(f) \leq h^k_\top(f) + \frac kr \lip(f).
 $$
\end{theorem}

\begin{remark}
The above extra term is indeed necessary as shown already by examples attributed by Yomdin
\cite{Yomdin1} to Margulis: there is $f:[0,1]\to[0,1]$, $C^r$ with $h_\top(f)=0$ and $\gamma(f)=
\lip(f)/r$.
\end{remark}

\begin{remark}
Yomdin's estimate is \emph{uniform} holding for each disk and each iterate. Its proof involves very little dynamics and no ergodic theory, in contrast to Newhouse's inequality quoted above.
\end{remark}

\begin{corollary}
Let $f:M\to M$ be a self-map of a compact manifold. If $f$ is $C^\infty$, then
 $$
   h_\top(f) = \gamma(f).
 $$
\end{corollary}

Let $f_*:H_*(M,\RR)\to H_*(M,\RR)$ be the total homological action of $f$.
Let $\rho(f_*)$ be its spectral radius. As the $\ell^1$-norm in homology
gives a lower bound on the volume, we have $\gamma(f)\geq \log \rho(f_*)$.
Hence, the following special case of the Shub Entropy Conjecture 
is proved:

\begin{corollary}[Yomdin \cite{Yomdin1}] 
Let $f:M\to M$ be a self-map of a compact manifold. If $f$ is $C^\infty$, then
 $$
   \log \rho(f_*) \leq h_\top(f).
 $$
\end{corollary}

\subsection{Resolution entropies}

The previous results of Yomdin and more can be obtained by computing a growth
rate taking into account the full structure of singular disks. A variant of this idea is
explained in Gromov's Bourbaki Seminar \cite{Gromov} on Yomdin's results.
We build on \cite{BuzziPhD}.

\begin{definition}
Let $r\geq1$. Let $\phi:Q^k\to M$ be a $C^r$ singular $k$-disk. A \new{$C^r$-resolution} $\mathcal R$ of order $n$ of
$\phi$ is a collection of $C^r$ maps $\psi_\omega:Q^k\to Q^k$, for $\omega\in\Omega$ with
$\Omega$ a finite collection of words of length at most $n$ with the following properties.
For each $\omega\in\Omega$, let $\Psi_\omega:=\psi_{\sigma^{|\omega|-1}\omega}\circ\dots\circ\psi_\omega(Q^k)$.
We require:
  \begin{enumerate}
   \item $\bigcup_{|\omega|=n} \Psi_\omega(Q^k)=Q^k$;
   \item $\|\psi_\omega\|_{C^r}\leq 1$ for all $\omega\in\Omega$;
   \item $\|f^{|\omega|}\circ\Psi_\omega\|_{C^r}\leq 1$ for all $\omega\in\Omega$,
 \end{enumerate}
The \new{size} $|\mathcal R|$ of the resolution is the number of words in $\Omega$
with length $n$.
\end{definition}

Condition (2) added in \cite{BuzziPhD} much simplifies the link between resolutions 
and entropy. It no longer relies on Newhouse application of Pesin theory and becomes
straightforward:

\begin{fact}\label{fact:resolution-cover}
Let $\mathcal R:=\{\psi_\omega:Q^k\to Q^k:\omega\in\Omega\}$ be a $C^r$-resolution of order
$n$ of $\phi:Q^k\to M$. Let $\eps>0$ and $Q^k_\eps$ be $\eps$-dense in $Q^k$, i.e., 
$Q^k\subset \bigcup_{t\in Q^k_\eps} B(x,\eps)$. Then
 $$
    \{\Psi_\omega(t):t\in Q^k_\eps\text{ and }\omega\in\Omega\text{ with }|\omega|=n\}
      \text{ is a $(\eps,n)$-cover of } \phi(Q^k).
 $$
\end{fact}

On the other hand, the notion of resolution induces entropy-like quantities:

\begin{definition}
Let $1\leq r<\infty$ and let $f:M\to M$ be a $C^r$ self-map of a compact
manifold.
Let $R_f(C^r,n,\phi)$ be the minimal size of a $C^r$-resolution of
order $n$ of a $C^r$ singular disk $\phi$.
The \new{resolution entropy} of $\phi$ is:
 $$
    h_R(f,\phi):=\limsup_{n\to\infty} \frac1n\log R_f(C^r,n,\phi).
 $$
If $\mathcal D$ is a collection of $C^r$ singular disks, its $C^r$ resolution entropy
is
 $$
    h_{R,C^r}(f,\mathcal D) := \sup_{\phi\in\mathcal D} h_R(f,\phi)
 $$
and its $C^r$ uniform resolution entropy is:
 $$
    H_{R,C^r}(f,\mathcal D) := \limsup_{n\to\infty} \frac1n \log R_f(C^r,n,\mathcal D)
 $$
where $R_f(C^r,n,\mathcal D):= \max_{\phi\in\mathcal D} R_f(C^r,n,\phi)$. We set:
 $$
   h^{k,C^r}_R(f)=h_{R,C^r}(f,\mathcal D^k_r) \text{ and }
   H^{k,C^r}_R(f)=H_{R,C^r}(f, D^k_r).
 $$
\end{definition}

The following is immediate but very important:

\begin{fact}\label{fact:Rres-sub-mult}
Let $1\leq r<\infty$ and $0\leq k\leq d$.
Let $f:M\to M$ be a $C^r$ self-map of a compact $d$-dimensional manifold.
The sequence $n\mapsto R_f(C^r,n,D^k_r)$ is sub-multiplicative:
 $$
   R_f(C^r,n+m,D^k_r) \leq R_f(C^r,n,D^k_r)R_f(C^r,m,D^k_r).
 $$ 
\end{fact}

The key technical result of Yomdin's theory can be formulated as follows:

\begin{proposition}\label{prop:CardRes}
Let $1\leq r<\infty$ and $\alpha>0$. Let $f:M\to M$ be a $C^r$ self-map of a compact manifold.
There exist constants $C',C(r,\alpha),\eps_0(r,\alpha)$ with the following property.
For any $C^r$ singular disk $\phi$, any number $0<\eps<\eps_0(r,\alpha)$ and any 
integer $n\geq1$, 
 $$
   C'\eps^k r_f(\eps,n,\phi) \leq R_f(C^r,n,\phi)
      \leq C(r,\alpha) \Lip(f)^{(\frac kr +\alpha)n} r_f(\eps,n,\phi).
 $$
\end{proposition}

Remark that the above constants depend on $f$. The first inequality follows
from Fact \ref{fact:resolution-cover}. The second is the core of Yomdin theory,
we refer to \cite{BuzziPhD} for details.

%
%
%
%

\section{Properties of Dimensional Entropies}

We turn to various properties of dimensional entropies, most of which can
be shown using resolution entropy and its submultiplicativity.

\subsection{Link between Topological and Resolution Entropies}

We start by observing that Proposition \ref{prop:CardRes} links the topological and resolution
entropies.

\begin{corollary}
For all positive integers $r,k$, any collection of $C^r$ $k$-disks $\mathcal D$ and any $C^r$ self-map
$f$ on a manifold  equipped with a $C^r$ size:
 \begin{align*}
   h_\top(f,\mathcal D) \leq h_{R,C^r}(f,\mathcal D) \leq  h_\top(f,\mathcal D) + \frac kr \log\Lip(f)\\
   H_\top(f,\mathcal D) \leq H_{R,C^r}(f,\mathcal D) \leq  H_\top(f,\mathcal D) + \frac kr \log\Lip(f).
 \end{align*}
\end{corollary}

If the disks in $\mathcal D$ are $C^\infty$, then, for $r\leq s<\infty$,
 $$
   h_{R,C^r}(f,\mathcal D) \leq h_{R,C^s}(f,\mathcal D) \leq h_\top(f,\mathcal D) + \frac ks \log\Lip(f).
 $$
Letting $s\to\infty$, we get:

\begin{corollary}\label{cor:hR-htop}
If $f$ is $C^\infty$, then, for all $1\leq r<\infty$,
 $$
    h_{R,C^r}(f,\mathcal D^k_\infty) =  h_\top(f,\mathcal D^k_\infty).
 $$
The same holds for uniform topological entropy.
\end{corollary}

\subsection{Gap between Uniform and Ordinary Dimensional Entropies}

Yomdin theory gives the following relation:

\begin{proposition}\label{prop:h-and-H}
Let $1\leq r<\infty$ and $f:M\to M$ be a $C^r$ self-map of a compact $d$-dimensional manifold.
For each $0\leq k\leq d$,
 $$
    h^{k,C^r}_\top(f) \leq H^{k,C^r}_\top(f) \leq h^{k,C^r}_\top(f)+\frac kr \lip(f)
 $$
and the same holds for the resolution entropies $h^{k,C^r}_R(f)$ and $H^{k,C^r}_R(f)$.

In particular, in the $C^\infty$ smooth case, all the versions of the dimensional entropies agree: $h^k_\top(f)=H^k_\top(f)=h^k_R(f)=H^k_R(f)=\lim_{r\to\infty} H^{k,C^r}_\top(f)$.
\end{proposition}

\begin{proof}
It is obvious that the uniform entropies dominate ordinary ones. By Fact \ref{fact:resolution-cover},
$h^{k,C^r}_\top(f)\leq h^{k,C^r}_R(f)$ and $H^{k,C^r}_\top(f)\leq H^{k,C^r}_R(f)$. Therefore it
is enough to show:
 \begin{equation}\label{eq:proof-unif-ent}
    H^{k,C^r}_R(f) \leq h^{k,C^r}_\top(f)+\frac kr \lip(f).
 \end{equation}
Let $\alpha>0$. Let $\eps_0>0$ as in Proposition
\ref{prop:CardRes}. This proposition defines a number $C(r,\alpha)$.
By definition, for every $\phi\in D^k_r$, there exists $n_\phi<\infty$ such that
$$r_f(\eps_0,n_\phi,\phi)\leq e^{(h^{k,C^r}_\top(f)+\alpha)n_\phi}.$$
We can arrange it so that this holds for all $k$-disks $\psi$ in some
$C^0$ neighborhood $\mathcal U_\phi$ of $\phi$. We also assume $n_\phi$ so large
that $C(r,\alpha)\leq e^{\alpha n_\phi}$.

By Proposition \ref{prop:CardRes}, each such $\psi$ admits a resolution with size at most
 $$
  r_f(\eps,n_\phi,\psi)\times C(r,\alpha) \Lip(f)^{\frac kr n_\phi}
    \leq e^{(h^{k,C^r}_R(f)+\frac kr\lip(f)+ 2\alpha)n_\phi}.
 $$

$D^k_r$ is relatively compact in the $C^0$ topology, hence there is a finite cover
$D^k_r\subset \mathcal U_{\phi_1}\cup\dots\cup\mathcal U_{\phi_K}$. Let $N:=\max n_{\phi_j}$.
It is now easy to build, for each $n\geq0$ and each $\psi\in D^k_r$ a $C^r$ resolution
$\mathcal R$ of order $n$ with:
 $$
   |\mathcal R| \leq \exp(h^{k,C^r}_\top(f)+\frac kr\lip(f)+ 2\alpha)(n+N).
 $$
(\ref{eq:proof-unif-ent}) follows by letting $\alpha$ go to zero.
\end{proof}

\subsection{Continuity properties}

\begin{proposition}
We have the following upper semicontinuity properties:
 \begin{enumerate}
  \item $f\mapsto H^{k,C^r}_R(f)$ is upper semicontinuous in the $C^r$ topology for all $1\leq r<\infty$;
  \item $f\mapsto H^k_\top(f)$ is upper semicontinuous in the $C^\infty$ topology;
  \item the defect in upper semi-continuity of $f\mapsto H^{k,C^r}_\top(f)$ at $f=f_0$ is at most
$\frac{k}{r}\lip(f_0)$:
 $$
    \limsup_{f\to f_0} H^{k,C^r}_\top(f) \leq H^{k,C^r}_\top(f_0)+\frac{k}{r}\lip(f_0).
 $$
 \end{enumerate}
\end{proposition}

\begin{proof}
We prove (1).
The sub-multiplicativity of resolution numbers observed in Fact \ref{fact:Rres-sub-mult}
implies that: $H_R^{k,C^r}(f) = \inf_{n\geq1} \frac1n \log R_f(C^r,n,D^k_r)$. 
For each fixed positive integer $n$, $R_g(C^r,n,D^k_r)\leq 2^k R_f(C^r,n,D^k_r)$ for
any $g$ $C^r$-close to $f$ (use a linear subdivision). Thus $f\mapsto H_R^{k,C^r}(f)$
is upper semi-continuous in the $C^r$ topology.

We deduce (3) from (1).
Let $f_n\to f$ in the $C^r$ topology. By the preceding, $H_R^{k,C^r}(f)\geq\limsup_{n\to\infty}
H_R^{k,C^r}(f_n)$. By Proposition \ref{prop:h-and-H}, $H_\top^k(f)\geq H_R^{k,C^r}(f)-\frac kr \lip(f)$.

(2) follows from (3) using Lemma \ref{lem:h-Cinfty}. 
\end{proof}

On the other hand, $f\mapsto H^k_\top(f)$ fails to be lower semi-continuous except for interval
maps for which topological entropy is lower semi-continuous in the $C^0$ topology by a result
of Misiurewicz. In every case there are counter examples:

\begin{example}
For any $d\geq2$ and $1\leq k\leq d$, there is a self-map of a compact manifold of dimension $d$ 
at which $h^k_\top(f)$ fails to be lower semi-continuous.

Let $h:\RR\to[0,1]$ be a $C^\infty$ function such that $h(t)=1$ if and only if $t=0$.
Let $F_\lambda:[0,1]^d\to[0,1]^d$ be defined by
 $$
   F_\lambda(x_1,\dots,x_d)=(h(\lambda)x_1,4x_1 x_2(1-x_2),x_3,\dots,x_d)
 $$
Observe that if $\lambda\ne0$, then $h(\lambda)\in[0,1)$ and $F_\lambda^n(x_1,\dots,x_d)$
approaches $\{(0,0)\}\times[0,1]^{d-2}$ on which $F_\lambda$ is the identity. Therefore
$h_\top(F_\lambda)=0$. On the other hand $h_\top(F_0)=h_\top(x\mapsto 4x(1-x))=\log 2$.
Now, $H^k_\top(F_\lambda)\leq h_\top(F_\lambda)=0$ for any $\lambda\ne 0$ and
$H^k_\top(F_0)\geq h^k_\top(f)\geq h_\top(F_0,\{1\}\times[0,1]\times\{(0,\dots,0)\})=\log 2$ for any
$k\geq1$.
\end{example}

\section{Hyperbolicity from Entropies}

We now explain how the dimensional entropies can yield
dynamical consequences. We start by recalling an inequality
which will yield hyperbolicity at the level of measures.
Then we give the definition and main results for entropy-expanding
maps. Finally we explain the new notion of entropy-hyperbolicity
for diffeomorphisms.

\subsection{A Ruelle-Newhouse type inequality}

One of the key uses of dimensional entropies is to give
bounds on the exponents using the following estimate. 
This will give hyperbolicity of large entropy measure from
assumptions on these dimensional entropies.

\begin{theorem}\label{thm:RN-inequality} \cite{BuzziEE}
Let $f:M\to M$ be a $C^r$ self-map of a compact manifold with $r>1$.
Let $\mu$ be an ergodic, invariant probability measure with
Lyapunov exponents $\lambda_1(\mu)\geq\lambda_2(\mu)\geq\dots\geq\lambda_d(\mu)$
repeated according to multiplicity. Recall that $H^k_{\top}(f)$  is the
uniform $k$-dimensional entropy of $f$. Then:
 $$
   h(f,\mu) \leq H^k_\top(f)+\lambda_{k+1}(\mu)^++\dots+\lambda_d(\mu)^+.
 $$
\end{theorem}

\begin{remark}
For $k=0$ this reduces to Ruelle's inequality. For $k$ equal to the
number of nonnegative exponents, this is close to Newhouse inequality
(with $H^k_\top(f)$ replacing $\gamma^k(f)$). The proof is similar to
Newhouse's and relies on Pesin theory.
\end{remark}

\subsection{Entropy-expanding Maps}

We require that the full topological entropy only appear at
the full dimension.

\begin{definition}
A $C^r$ self-map $f:M\to M$ of a compact manifold is \new{entropy-expanding}
if:
 $$
    H^{d-1}_\top(f) < h_\top(f).
 $$
\end{definition}

An immediate class of examples is provided by the interval
maps with non-zero topological entropy.

\medbreak

The first consequence of this condition is that ergodic invariant
probability measures with entropy $>H^{d-1}_\top(f)$ have only
Lyapunov exponents bounded away from zero. This follows immediately
from Theorem \ref{thm:RN-inequality}. 

This also allows the application
of (a non-invertible version of) Katok's theorem, proving a logarithmic
lower bound on the number of periodic points.

Katok's proof gives 
horseshoes with topological entropy approaching $h_\top(f)$. In particular
these maps are points of lower semi-continuity of $f\mapsto h_\top(f)$
in any $C^r$ topology, $r\geq0$.
Combining with the upper semi-continuity from Yomdin theory we get:

\begin{proposition}\label{prop:hexp-open}
The entropy-expansion property is open in the $C^\infty$ topology.
\end{proposition}

Thus we can use the following estimate

\begin{proposition} \cite{BuzziSMF}
The Cartesian product of a finite number of $C^\infty$ smooth interval maps,
each with nonzero topological entropy is entropy-expanding.
\end{proposition}

To get dynamically interesting examples:

\begin{example}
For $|\eps|$ small enough, the plane map $F_\eps:(x,y)\mapsto(1-1.8x^2-\eps y^2 ,1-1.9y^2-\eps x^2)$
preserves $[-1,1]^2$ and its restriction to this set is entropy-expanding.
\end{example}

A sufficient condition, considered in a different approach by Oliveira and Viana \cite{Oliveira1,
Oliveira2} is the following:

\begin{lemma}
Let $f:M\to M$ be a diffeomorphism of a compact Riemanian manifold. Let
$\|\Lambda^kTf\|$ be the maximum over all $1\leq \ell\leq k$ and all $x\in M$
of the Jacobian of the restrictions of the differential $T_xf$ to any
$k$-dimensional subspace of $T_xM$. Then
 $$
    H^k_\top(f) \leq \log\|\Lambda^kTf\|.
 $$
In particular, $\log\|\Lambda^kT\|<h_\top(f)$ implies that $f$ is 
entropy-expanding. An even stronger condition is $(d-1)\lip(f)<h_\top(f)$.
\end{lemma}

The proof of this lemma is a variation on the classical proof of Ruelle's inequality.

\medbreak

We are able to analyze the dynamics of entropy-expanding maps with respect to large entropy measures
rather completely. 

\begin{theorem}\label{thm-hexp}
Let $f:M\to M$ be a $C^\infty$ self-map of a compact manifold. Assume that $f$
is entropy-expanding. Then:
 \begin{itemize}
  \item $f$ has finitely many maximum measures;
  \item its periodic points satisfies a multiplicative lower bound.
 \end{itemize}
\end{theorem}

This can be understood as generalization of the Markov
property which corresponds to partition having boundaries with essentially
finite forward or backward orbits. The proof of the theorem involves a partition
whose boundaries are pieces of smooth submanifolds, therefore of entropy
bounded by $H^{d-1}_\top(f)$.

In \cite{BuzziPQFT}, we are able to define a nice class of symbolic
systems, called \emph{puzzles of quasi-finite type}, which contains 
the suitably defined symbolic representations of
entropy-expanding maps satisfying a technical condition and
have the above properties. Moreover, their periodic points define zeta functions 
with meromorphic extensions and their natural extensions can be
classified up to entropy-conjugacy in the same way as interval maps.

\subsection{Entropy-Hyperbolicity}

Entropy-expanding maps are \emph{never diffeomorphisms}. Indeed,
they have ergodic invariant measures which have nonzero entropy
and only positive Lyapunov exponents. Wrt the {inverse diffeomorphism}
these measures have the same nonzero entropy but only negative
Lyapunov exponents, contradicting Ruelle's inequality.
Thus we need a different notion for diffeomorphism.

\begin{definition}
The \new{unstable (entropy) dimension} is:
 $$
   d_u(f):=\min\{0\leq k\leq d : H^k_\top(f)=h_\top(f) \}.
 $$
If $f$ is a diffeomorphism, then the \new{stable dimension} is: $d_s(f):=d_u(f^{-1})$
(if $f$ not a diffeomorphism we set $d_s(f)=0$).
\end{definition}

Observe that $f$ is entropy-expanding if and only if $d_u(f)$ coincides with
the dimension of the manifold.

\begin{lemma}
Let $f:M\to M$ be a $C^r$ self-map of a compact $d$-dimensional manifold with $r>1$. Then:
 $$
   d_u(f)+d_s(f) \leq d.
 $$
\end{lemma}

\begin{proof}
Theorem \ref{thm:RN-inequality} implies that measures with entropy $>H^{d_u(f)-1}_\top(f)$ have
at least $d_u(f)$ positive exponents. The same reasoning
applied to $f^{-1}$ shows that such measures have at least $d_s(f)$ negative exponents. 
By the variational principle such measures exist. Hence $d_u(f)+d_s(f)\leq d$.
\end{proof}

We can now propose our definition:

\begin{definition}
A diffeomorphism such that $d_u(f)+d_s(f)=d$ is \new{entropy-hyperbolic}.
\end{definition}

Obviously surface diffeomorphisms with non-zero topological entropy
are entropy-hyperbolic.

\medbreak

Exactly as above, we obtain from Theorems \ref{thm:RN-inequality} and \ref{thm:Katok}:

\begin{theorem}
Let $f:M\to M$ be a $C^r$ diffeomorphism of some compact manifold with $r>1$.
Assume that $f$ is entropy-hyperbolic. Then:
 \begin{itemize}
  \item all ergodic invariant probability measures with entropy close enough to
the topological entropy have the absolute value of their Lyapunov exponents
bounded away from zero;
  \item their periodic points satisfy a logarithmic lower bound;
  \item they contain horseshoes with topological entropy arbitrarily close to
that of $f$.
 \end{itemize}
\end{theorem}

\begin{corollary}
The set of entropy-hyperbolic diffeomorphisms of a compact manifold
is open in the $C^\infty$ topology.
\end{corollary}

\subsection{Examples of Entropy-Hyperbolic Diffeomorphisms}
The techniques of \cite{BuzziSMF} yield:

\begin{lemma}
Linear toral automorphisms are entropy-hyperbolic
if and only if they are hyperbolic in the usual sense: no eigenvalue
lies on the unit circle.
\end{lemma}

The following condition is easily seen to imply entropy-hyperbolicity:

\begin{lemma}\label{lem:unif-hh}
Let $f:M\to M$ be a diffeomorphism of a compact manifold of dimension
$d$. Assume that there are two integers $d_1+d_2=d$ such that:
 $$
   \log\|\Lambda^{d_1-1}Tf\|<h_\top(f) \text{ and }
   \log\|\Lambda^{d_2-1}Tf\|<h_\top(f).
 $$
Then $f$ is entropy-hyperbolic.
\end{lemma}

\section{Further directions and Questions}

We discuss some developping directions and ask some questions.

\subsection{Variational Principles}

It seems reasonable to conjecture the following \emph{topological variational principle} for dimensional entropies,
at least for $C^\infty$ self-maps and diffeomorphisms:

{In each dimension, there is a $C^\infty$ disk with maximum topological entropy, i.e., $h^k_\top(f)$.}

Does it fail for finite smoothness?

\medbreak

A probably more interesting but more delicate direction would be an \emph{ergodic variational
principle}. Even its formulation is not completely clear. A possibility would be as follows:

For each dimension $k$, $h^k_\top(f)$ is the supremum of the entropies of $k$-disks contained
in unstable manifolds of points in any set of full measure with respect to all invariant 
probability measures.

\subsection{Dimensional Entropies of Examples}

Let $f_i:M_i\to M_i$ are smooth maps for $i=1,\dots,n$ and consider the following formula:
 $$
    H^k_\top(f_1\times\dots\times f_n) = \max_{\ell_1+\dots+\ell_n=k}
      H^{\ell_1}_\top(f)+\dots+H^{\ell_n}_\top(f).
 $$
This is only known in special cases --see \cite{BurguetPhD}. It would imply that product
of entropy-expanding maps are again entropy-expanding.

\medbreak

If $f:M\to M$ is an expanding map of a compact manifold, is it true that $H^{d-1}_\top(f)<h_\top(f)$.
Note that this fails for piecewise expanding maps (think of a limit set containing an isolated invariant
curve with maximum entropy). 

Likewise is an Anosov diffeomorphism, even far from linear,
entropy-hyperbolic?

\medbreak

Find examples where $h^{k,C^r}_\top(f)<H^{k,C^r}_\top(f)$.

\subsection{Other types of dimensional complexity}

Other "dimensional complexities" have been investigated from growth rates of multi(co)vectors
for the Kozlovski entropy formula \cite{Kozlovski} to the currents which are fundamental to
multidimensional complex dynamics \cite{Gromov2} and the references therein.

How do they relate to the above dimensional entropies?

\subsection{Necessity of Topological Assumptions}

We have seen in Sect. 2.1 that, for maps, the assumption of no zero Lyapunov exponent for the large entropy measure, (or even that these exponents are bounded away from zero) is not enough for our purposes (e.g., finiteness of the number of maximum measures). Such results seem to require more uniform assumptions,
like the one we make on dimensional entropies.

Is it the same for diffeomorphisms? That is, can one find diffeomorphisms  with infinitely many
maximum measures, all with exponents bounded away from zero?

Let $f$ be a $C^r$ self-map of a compact manifold. Assume that there are numbers
$h<h_\top(f)$, $\lambda>0$, such that the Lyapunov exponents of any ergodic
invariant measure with entropy at least $h$ fall outside of $[-\lambda,\lambda]$.
Assume also that the set of invariant probability measures with entropy $\geq h$
is compact. Does it follow that there are only finitely many maximum measures?

\subsection{Entropy-Hyperbolicity}

In a work in progress with T. Fisher, we show that the  condition of Lemma \ref{lem:unif-hh}
is satisfied by a version of a well-known example of robustly transitive, non-uniformly hyperbolic
diffeomorphism of $\TT^4$ due to Bonatti and Viana \cite{BUH}. Building nice
center-stable and center-unstable invariant laminations, we expect be able
to show the same properties as in Theorem \ref{thm-hexp}. 

\medbreak

I however conjecture that the finite number of maximum measures, etc. should in fact
hold for every $C^\infty$ entropy-expanding diffeomorphism, even when there is no such
nice laminations. Of course this contains the case of surface diffeomorphisms which
is still open (see Conjecture \ref{conj:surface}), despite the result on a toy model \cite{BuzziPWAH}.

\subsection{Generalized Entropy-Hyperbolicity}

It would be interesting to have a \emph{more general notion of entropy-hyperbolicity}. For instance,
if a hyperbolic toral automorphism is entropy-hyperbolic, this is not the case for the disjoint
union of two such systems of the same dimenion if they have distinct stable dimensions. It may
be possible to "localize" the definition either near points or near invariant measures to avoid
these stupid obstructions (this is one motivation for the above question on variational
principles for dimensional entropies).

If one could remove such obstructions, the remaining ones could reflect basic dynamical phenomena
opening the door to a speculative "entropic Palis program".

\appendix

\section{$C^r$ sizes}\label{appendix:Cr-size}

We explain how to measure the $C^r$ size of singular disks of a compact manifold $M$
of dimension $d$. Here $1\leq r\leq\infty$.

We select a finite atlas $\mathcal A$ made of charts $\chi_i:U_i\subset\RR^d\to M$ such that
changes of coordinates $\chi_i^{-1}\circ\chi_j$ are $C^r$-diffeomorphisms of open subsets
of $\RR^d$. Then we define, for any singular $k$-disk $\phi$:
 $$
   \| \phi \|_{C^r} := \sup_{x\in M} \inf_{U_i\ni x} \max_{s_1+\dots+s_k\leq r} \max_{1\leq j\leq d}
        |\partial_{t_1}^{s_1}\dots\partial_{t_k}^{s_k}(\pi_j\circ\chi_i\circ\phi)(t_1,\dots,t_k)|
 $$
where the above partial derivatives are computed at $t=\chi_i^{-1}(x)$ and $\pi_j(u_1,\dots,u_d)=u_j$.

\begin{fact}\label{fact:equiv-Cr}
If $\|\cdot\|_{C^r}$ and $\|\cdot\|'_{C^r}$ are two $C^r$ size defined by the above
procedure, there exists a constant $K$ such that, for any $C^r$ $k$-disk $\phi:Q^k\to M$:
 $$
    \|\phi\|_{C^r}\leq K\cdot \|\phi\|'_{C^r}.
 $$
\end{fact}

\begin{fact}\label{fact:Cinfty-approx}
Let $\phi:Q^k\to M$ be a $C^r$ disk for some finite $r\geq1$. Then, for any $t_0\in Q^k$,
there exists a $C^\infty$ approximation $\phi_\infty:Q^k\to M$ such that:
 $$
   \forall t\in Q^k\quad d(\phi_\infty(t),\phi(t)) \leq \|\phi\|_{C^r} \|t-t_0\|^r.
 $$
and $\|\phi_\infty\|_{C^\infty}\leq 2 \|\phi\|_{C^r}$.
\end{fact} 

This is easily shown by considering a neighborhood of $\phi(t_0)$ contained in a single
chart of $\mathcal A$ and approximating $\phi$ by its Taylor expansion in that chart.

\end{document}